%% file: main.tex
\documentclass[a4paper,11pt,leqno]{amsart}

\input{Preamble}

\title{Random perfect matchings in regular graphs}

\begin{document}

\begin{abstract}
	We prove that in all regular robust expanders $G$, every edge is asymptotically equally likely contained in a uniformly chosen perfect matching $M$.
	We also show that given any fixed matching  or spanning regular graph $N$ in $G$, 
	the random variable $|M\cap E(N)|$ is approximately Poisson distributed.
	This in particular confirms a conjecture and a question due to Spiro and Surya, and complements results due to Kahn and Kim who proved that in a regular graph every vertex is asymptotically equally likely contained in a uniformly chosen matching.
	Our proofs rely on the switching method and the fact that simple random walks mix rapidly in robust expanders.
\end{abstract}

\maketitle

\section{Introduction}

A remarkable result due to Kahn and Kim~\cite{kahn1998random} says that in \emph{any} $d$-regular graph $G$,
the probability that a vertex is contained in a uniformly chosen matching in $G$ is $1-(1+o_d(1))d^{-\frac{1}{2}}$.
This shows that the structure of a $d$-regular graph has essentially no impact on the probability that a vertex is contained in a uniformly chosen matching.

In this paper we are interested in uniformly chosen \emph{perfect} matchings.
Then, surely, each vertex is contained in every perfect matching.
Hence, as the statement for vertices is trivial, what about the probability that an edge is contained in a random perfect matching?
Is each edge equally likely contained in a random perfect matching?
A moment of thought reveals that this is wrong in a very strong sense.
In every odd-regular graph with exactly one bridge%
	\COMMENT{Let $G$ be an odd-regular graph and $e$ be a bridge, joining two components $C$ and $C'$ of $G-e$. The endpoints of $e$ are the only vertices of even degree in $G-e$, so $|C|$ and $|C'|$ must be odd.},
the bridge is contained in every perfect matching, while the edges adjacent to the bridge are contained in none of the perfect matchings.
Therefore, in order to avoid a trivial statement further conditions are needed.

Hall's condition for the existence of perfect matchings in bipartite graphs says that the neighbourhood of an (independent) set should be at least as large as the set itself, which is clearly also a necessary condition.
Here, we assume that this property is present in a robust sense in order to avoid the trivial scenarios mentioned above.
More precisely,
let $\nu,\tau>0$ and $G$ be a graph on $n$ vertices.
Then, we define the \emph{$\nu$-robust neighbourhood $RN_{\nu,G}(S)$} of a set $S\subseteq V(G)$ in $G$ to be the set of vertices of $G$ which have at least $\nu n$ neighbours in $S$. 
We say that $G$ is a \emph{robust $(\nu, \tau)$-expander} if $|RN_{\nu,G}(S)| \geq  |S| + \nu n$ for each 
$S\subseteq V(G)$ satisfying $\tau n \leq |S| \leq  (1 - \tau)n$.
Robust expansion is a fairly mild assumption and consequently it proved to be useful in several situations, see for example~\cite{kuhn2013hamilton,gruslys2021cycle,kuhn2015robust}.

We denote by $\cP(G)$ the set of all perfect matchings in $G$
and write $M\sim U(\cP(G))$ to refer to a uniformly chosen matching from $\cP(G)$.
Our main result implies that such matchings $M$ are extremely well-distributed in robust expanders.

\begin{thm}\label{thm:edge}
	For any $\delta>0$, there exists $\tau>0$ such that for all $\nu>0$, there exists $n_0\in \mathbb{N}$ for which the following holds. Let $n\geq n_0$ be even and $d\geq \delta n$. 
	Then, for any $d$-regular robust $(\nu, \tau)$-expander~$G$ on $n$ vertices, $M\sim U(\cP(G))$, and $e\in E(G)$, we have
	\begin{align*}
		\bP[e\in M]= (1+o_n(1))d^{-1}.
	\end{align*}
\end{thm}

In fact much more is true. 
Fix any matching $N$ in $G$, let $M\sim U(\cP(G))$, and consider $X\coloneqq |M\cap N|$.
Then, linearity of expectation and \cref{thm:edge} imply that $\bE[X]=(1+o_n(1))d^{-1}|N|$.
Employing the heuristic that each edge is \emph{independently} present in $M\sim U(\cP(G))$ with probability~$d^{-1}$,
then we expect that $X$ has a binomial distribution with parameters $|N|$ and $d^{-1}$.
This is approximated by a Poisson distribution with parameter $d^{-1}|N|$, whenever $|N|$ grows with $n$.
Our next result confirms this.

To this end, we define the \emph{total variation distance} of two integer-valued random variables $Y$ and $Z$ as $d_{\rm TV}(Y,Z)\coloneqq \frac{1}{2}\sum_{k\in \mathbb{Z}}|\bP[Y=k]-\bP[Z=k]|$, which measures how close two distributions are. Moreover, we write $Y\sim \Po(\lambda)$ if $Y$ is a random variable which follows a Poisson distribution with parameter $\lambda$.

\begin{thm}\label{thm:main}
	For any $\delta> 0$, there exists $\tau >0$ such that for all $\nu>0$, there exists $n_0\in \mathbb{N}$ for which the following holds. 
	Let $n\geq n_0$ be even and $d\geq \delta n$. Then, for any $d$-regular robust $(\nu, \tau)$-expander $G$ on $n$ vertices, $M\sim U(\cP(G))$, any matching $N$ in $G$, $X\coloneqq |M\cap N|$, and $Y\sim \Po(d^{-1}|N|)$, we have 
	$d_{\rm TV}(X,Y)=o_n(1)$. 
\end{thm}

The fact that $N$ is a matching is not crucial for our argument, however note for example that if $N$ is a star, then $X$ is a $\{0,1\}$-valued random variable. Hence, $X$ can only converge to a Poisson distribution if $N$ is somewhat spread out.
In particular, when $N$ is a spanning $r$-regular graph for some fixed $r$, we can derive an analogue of \cref{thm:main} (see \cref{sec:conclusion}), which answers a question of Spiro and Surya \cite{spiro2022counting}.

\Cref{thm:main} has some interesting consequences.
We define $\pma(G)\coloneqq |\cP(G)|$ and suppose~$G$ and~$M$ are as in Theorem~\ref{thm:main}.
Let $N$ be a perfect matching in $G$.
Then, \cref{thm:main} implies that
\begin{align*}\label{eq:pm}
	\frac{\pma(G-N)}{\pma(G)}=\bP[M\cap N =\emptyset]=(1+o_n(1)) e^{-\frac{n}{2d}}.
\end{align*}
For a graph $G$ with a perfect matching, we denote by $G^\circ$ a subgraph of $G$ where one perfect matching is removed.
\NEW{Various combinatorial problems can be expressed as determining $\frac{\pma(G^\circ)}{\pma(G)}$. For example, when $G=K_{\frac{n}{2}, \frac{n}{2}}$, this ratio is equal to the probability that a random permutation of order $\frac{n}{2}$ is fixed-point-free, and it is well known that this probability equals $(1+o_n(1))e^{-1}$. The case when $G=K_n$ also has a combinatorial interpretation, see \cite{johnston2022deranged}.}

Let $K_{a\times b}$ denote the complete multipartite graph with $a$ parts, each of size $b$.
As an interpolation between the cases $K_{\frac{n}{2},\frac{n}{2}}$ and $K_n$, one may ask whether ${\pma(K_{r\times \frac{n}{r}}^\circ)}({\pma(K_{r\times \frac{n}{r}})})^{-1}$ converges to a limit.
Johnston, Kayll, and Palmer~\cite{johnston2022deranged} formulated this as a conjecture (and conjectured the limit value).
Recently this was resolved by Spiro and Surya~\cite{spiro2022counting}.
As all these graphs are robust expanders (excluding $K_{\frac{n}{2},\frac{n}{2}}$; we discuss bipartite graphs in \cref{sec:conclusion}), Theorem~\ref{thm:main} reproves the result due to Spiro and Surya~\cite{spiro2022counting}.

In fact, Spiro and Surya~\cite{spiro2022counting} also speculate whether for any $\alpha>\frac{1}{2}$, all regular graphs $G$ on an even number $n$ of vertices with $\delta(G)\geq \alpha n$ satisfy $\frac{\pma(G^\circ)}{\pma(G)}\to e^{-\frac{1}{2\alpha}}$, but consider this statement far too strong to be true. As it is trivial to show that graphs on $n$ vertices with $\delta(G)\geq (\frac{1}{2}+o_n(1))n$ are robust expanders, Theorem~\ref{thm:main} shows that this statement is actually true.

\bigskip

Our proof strategy is as follows. Let $G, M, N$, and $X$ be as in the statement of \cref{thm:main}.
We estimate the ratios of the form $\frac{\bP[X=k]}{\bP[X=k-1]}$ via the so-called switching method.
Knowing all relevant fractions of this type already exhibits the distribution of $X$, which has the advantage that the probabilities $\bP[X=k]$ do not need to be calculated directly.

The switching method is implemented as follows (see \cref{lm:ratio}). Fix a positive integer~$k$ and denote by $\cM_k$ and $\cM_{k-1}$ the sets of perfect matchings in $G$ which contain precisely $k$ and $k-1$ edges of $N$, respectively.
Then, construct an auxiliary bipartite graph $H$ on vertex classes~$\cM_k$ and $\cM_{k-1}$ by joining two perfect matchings $M\in \cM_k$ and $M'\in \cM_{k-1}$ if there is a cycle $C$ of length $2\ell$ in $G$
which contains precisely one edge of $N$ and alternates between edges of~$M$ and~$M'$.
(In other words, $M\in \cM_k$ and $M'\in \cM_{k-1}$ are adjacent in $H$ if $N\cap M'\subseteq N\cap M$ and the extra edge in $(N\cap M)\setminus M'$ can be `switched out' of $M$ to obtain $M'$ by exchanging $\ell$ edges of~$M$ for $\ell$ edges of $M'$, where these $2\ell$ edges altogether form a cycle.)

Note that if all perfect matchings in $\cM_k$ have degree (roughly) $d_k$ in $H$, while all perfect matchings in $\cM_{k-1}$ have degree (roughly) $d_{k-1}$, then $d_k|\cM_k|\approx e(H)\approx d_{k-1}|\cM_{k-1}|$. Hence, $\frac{\bP[X=k]}{\bP[X=k-1]}=\frac{|\cM_k|}{|\cM_{k-1}|}\approx\frac{d_{k-1}}{d_k}$.
Therefore, the crux of the proof consists in precisely estimating the number of such alternating cycles.

Counting the number of cycles of a certain length can be achieved using random walks as follows (see \cref{lm:countpaths}).
Given a $d$-regular graph, note that the number of walks of length $\ell$ starting at $u$ is precisely $d^\ell$, and so
the probability that a simple random walk that starts in $u$ is in $v$ after $\ell$ steps is equal to the number of walks from $u$ to $v$ of length~$\ell$ divided by $d^\ell$.
Since simple random walks are rapidly mixing in robust expanders, one can precisely estimate such probabilities, and therefore the number of such walks. 
A simple counting argument can eliminate those walks which are not paths, and so we can accurately count the number of cycles of fixed length in a regular robust expander. 
In practice, we have to consider simple random walks that use in every second step an edge from a fixed perfect matching~$M$. However, this additional technicality does not affect the mixing properties of such walks and so we can still precisely count them.

\bigskip

We remark that Spiro and Surya \cite{spiro2022counting} also used the switching method, which is common for this type of problems. 
Our contribution is to use longer cycles and perform the analysis with Markov chains; although the intuition is that the estimations become less precise with larger cycles, we employ key properties of Markov chains to show that in fact the \NEW{opposite} is true. Besides the fact that our results are substantially more general, the analysis also becomes significantly shorter and cleaner.

\section{Proof}

First, we introduce a directed version of robust expanders and state some results about such digraphs and random walks.
We then use these results to precisely count the number of walks of a specific length in regular robust expanders (\cref{lm:countpaths}).
This result plays a crucial role in the switching argument (\cref{lm:ratio}), from which \cref{thm:edge,thm:main} are derived.

\subsection{Robust outexpanders}

Robust outexpanders are the directed version of robust expanders.
Let $\nu,\tau>0$ and $D$ be a directed graph on $n$ vertices. 
The \emph{$\nu$-robust outneighbourhood $RN_{\nu, D}^+(S)$} of a set $S\subseteq V(D)$ in $D$ consists of all the vertices of $D$ which have at least $\nu n$ inneighbours in~$S$. We say that $D$ is a \emph{robust $(\nu, \tau)$-outexpander} if $|RN_{\nu, D}^+(S)|\geq |S|+\nu n$ for each $S\subseteq V(D)$ satisfying $\tau n \leq |S|\leq (1-\tau)n$.

We call a digraph $D$ on $n$ vertices \emph{$(\delta, f)$-almost regular} if all vertices of $D$ have in- and outdegree $(1\pm f)\delta n$.
(Obviously, we call a digraph \emph{regular} if it is $(\delta,0)$-regular for some $\delta$.)
It was observed in \cite{kuhn2013hamilton} that almost regular robust outexpanders contain, as a subdigraph, a regular robust outexpander of similar degree.

To state our results, we write $a\ll b$ for $a,b\in (0,1)$ if $a$ is sufficiently small in terms of $b$. Given an integer $n$, we denote $[n]\coloneqq \{1, \dots, n\}$ (in particular, $[n]=\emptyset$ if $n\leq 0$). 

\begin{lm}[{\cite[Lemma 5.2]{kuhn2013hamilton}}]\label{lm:regrob}
	Let $0<\frac{1}{n}\ll \varepsilon\ll \nu \leq \tau\ll \delta \leq 1$ and $\frac{1}{n}\leq f\leq \varepsilon$. 
	Let $D$ be a $(\delta, f)$-almost regular robust $(\nu, \tau)$-outexpander on $n$ vertices. 
	Then, $D$ contains a spanning $(1-f^{\frac{1}{2}})\delta n$-regular subdigraph which is a robust $(\frac{\nu}{2}, \tau)$-outexpander.
\end{lm}

\COMMENT{Note that the proof of \cite[Lemma 5.2]{kuhn2013hamilton} also holds if $\delta n^{-\frac{1}{2}}\leq \xi$ (rather than $\frac{1}{n}\ll \xi$). For each $x\in V(D)$, let $n_x^\pm \coloneqq d_D^\pm(x)-(1-f^{\frac{1}{2}})\delta n$ and note that $n_x^+,n_x^-=(1\pm f)\delta n-(1-f^{\frac{1}{2}})\delta n=(1\pm f^{\frac{1}{2}})f^{\frac{1}{2}}\delta n$.
Apply \cite[Lemma 5.2]{kuhn2013hamilton} with $1, f^{\frac{1}{2}}$, and $f^{\frac{1}{2}}\delta$ playing the roles of $q, \varepsilon$, and $\xi$ to obtain a spanning $D'\subseteq D$ where $d_{D'}^\pm(x)=n_x^\pm$ for each $v\in V(D)$. Then, $D\setminus D'$ is $(1-f^{\frac{1}{2}})\delta n$-regular, and is a robust $(\frac{\nu}{2}, \tau)$-outexpander since at most $3\sqrt{\varepsilon}n$ edges have been deleted at each vertex.}

The following lemma shows that in robust outexpanders on $n$ vertices there are $\Theta(n^{\ell-1})$ walks of length $\ell$ between any pair of vertices.

\begin{lm}\label{lm:walkrob}
	Let $0<\frac{1}{n}\ll \nu \leq \tau \leq 1$ and $\ell\geq \nu^{-1}+1$. Let $D$ be a robust $(\nu, \tau)$-outexpander on $n$ vertices with $d_D^+(v),d_D^-(v)\geq \tau n$ for each $v\in V(D)$. Then, $D$ contains at least $(\nu n)^{\ell-1}$ $(u,v)$-walks of length $\ell$ for any distinct $u,v\in V(D)$.           
\end{lm}

\begin{proof}
	Let $S_0\coloneqq \{u\}$ and $S_1\coloneqq N_D^+(u)$. For each $i\in [\ell-2]$, define $S_{i+1}\coloneqq RN_{\nu, D}^+(S_i)$. 
	For each $w\in N_D^-(v)\cap S_{\ell-1}$, there exist at least $(\nu n)^{\ell-2}$ $(u,w)$-walks $uu_1u_2\dots u_{\ell-2}w$ where $u_i\in S_i$ for each $i\in [\ell-2]$. 
	Since~$D$ is a robust $(\nu, \tau)$-outexpander with $d_D^+(w),d_D^-(w)\geq \tau n$ for each $w\in V(D)$, we have $|S_{i+1}|\geq \min\{(\tau+i\nu)n, (1-\tau+\nu)n\}$ for each $i\in [\ell-2]$. 
	In particular, $|N_D^-(v)\cap S_{\ell-1}|\geq \nu n$, so there are at least $(\nu n)^{\ell-1}$ $(u,v)$-walks $uu_1u_2\dots u_{\ell-1}v$ where $u_i\in S_i$ for each $i\in [\ell-1]$.
\end{proof}

\subsection{Random walks}

In the following we consider simple random walks $(X_t)_{t\geq 0}$ on digraphs~$D$;
that is, $\bP[X_{t+1}=v\mid X_t=u]=\frac{1}{d_D^+(u)}$ for each $u\in V(D)$ and $v\in N_D^+(u)$.

\NEW{As in the case of} regular graphs, the stationary distribution $(\sigma_v)_{v\in V(D)}$ of a regular robust outexpander $D$ on $n$ vertices is the uniform distribution; that is $\sigma_v\coloneqq n^{-1}$ for each $v\in V(D)$ (observe that random walks on robust outexpanders are aperiodic and irreducible by \cref{lm:walkrob}).

The next result yields a lower bound on the speed of convergence of a random walk to its stationary distribution.

\begin{lm}[{Joos and K\"uhn \cite[\NEW{Lemma 3.2}]{joos2022fractional}}]\label{lm:mixing}
	Let $(X_t)_{t\geq 0}$ be a Markov chain with state space $[n]$, transition matrix $P$ with $P(i,j)\neq 0$ for all $i,j\in [n]$, and (unique) stationary distribution given by $(\sigma_i)_{i\in [n]}$ with $\sigma_i\neq 0$ for all $i\in [n]$.
	Let $\alpha\coloneqq \min_{i,j,k\in [n]}\frac{P(i,j)}{\sigma_k}$ and $\beta\coloneqq \max_{i,j,k\in [n]}\frac{P(i,j)}{\sigma_k}$. Then, each $t\geq 2+2\alpha^{-1}\log \beta$ and $i\in [n]$ satisfy
	\[\mathbb{P}[X_t=i]=\Big(1\pm \left(1-\frac{\alpha}{2}\right)^t\Big)\sigma_i.\]
\end{lm}

Observe that all transition probabilities in \cref{lm:mixing} are non-zero.
In order to meet this assumption, we consider the Markov chain that is given by $(X_{\ell t})_{t\geq 0}$ for large enough $\ell$ where $(X_t)_{t\geq 0}$ is a simple random walk on a regular robust outexpander.

\begin{lm}\label{prop:markovrob}
	Let $0<\frac{1}{n}\ll \nu\leq \tau \ll \delta \leq 1$ and $\nu^{-1}+1 \leq k\leq  2\nu^{-1}$ and $c\in \mathbb{N}$.
	Let $D$ be a $\delta n$-regular robust $(\nu, \tau)$-outexpander on $[n]$. 
	Let $(X_t)_{t\geq 0}$ be the Markov chain corresponding to a simple random walk on $D$ and denote by $P$ its transition matrix. 
	Then, $(Y_t)_{t\geq 0}\coloneqq (X_{k t+c})_{t\geq 0}$ is a Markov chain with transition matrix $P^k$ and with the uniform distribution as unique stationary distribution $(\sigma_i)_{i\in [n]}$. Moreover,
	\[\nu^{k-1}\delta^{-k}\leq \frac{P^k(i,j)}{\sigma_\ell}\leq \delta^{-1}\]
	for any $i,j,\ell\in [n]$.
\end{lm}

\begin{proof}
	Note that $(\sigma_i)_{i\in [n]}=n^{-1}$
	is the unique stationary distribution for $(X_t)_{t\geq 0}$, and so for $(Y_t)_{t\geq 0}$ as well.	
	Let $i,j\in [n]$. By \cref{lm:walkrob},
	\begin{align*}
		P^k(i,j)
		&=\sum_{\ell_1, \dots, \ell_{k-1}\in [n]}P(i,\ell_1)P(\ell_1,\ell_2)\dots P(\ell_{k-1},j)\\
		&=\sum_{\substack{\ell_1, \dots, \ell_{k-1}\in [n]\colon \\ i\ell_1, \ell_1\ell_2, \dots, \ell_{k-1}j\in E(D)}}\frac{1}{d_D^+(i)d_D^+(\ell_1)\dots d_D^+(\ell_{k-1})}
		\geq \frac{(\nu n)^{k -1}}{\delta^k n^k} =\nu^{k-1}\delta^{-k} n^{-1}.
	\end{align*}
	Moreover,
	\begin{align*}
		P^k(i,j)&\NEW{=\mathbb{P}[X_{k+c}=j\mid X_c=i]}\\
		&=\mathbb{P}[X_{k+c-1}\in N_D^-(j)\NEW{\mid X_c=i}] \mathbb{P}[X_{k+c}=j\mid X_{k+c-1}\in N_D^-(j)\NEW{,X_c=i}]\\
		&\leq \sum_{\ell\in N_D^-(j)}\frac{\mathbb{P}[X_{k+c-1}=\ell\NEW{\mid X_c=i}]}{\delta n}\leq \NEW{1\cdot}\delta^{-1}n^{-1},
	\end{align*}
	which completes the proof.
\end{proof}

Due to the fact that simple random walks mix rapidly in expander graphs, the number of $(u,v)$-walks (and so $(u,v)$-paths) of length $\ell$ in robust outexpanders can be precisely estimated.
The next \lcnamecref{lm:countpaths}, which is the crucial tool for the proof of our main theorems, gives such an estimation.

For some technical reason, which becomes apparent in the proof of \cref{lm:ratio}, we only count $(u,v)$-paths which meet every edge of a fixed matching in at most one of its endpoints.
This has, however, no influence on the actual calculations made as this restriction is negligible.

\begin{lm}\label{lm:countpaths}
	Let $0<\frac{1}{n}\ll  \nu\leq \tau \ll \delta \leq 1$. 
	Fix an integer $\ell$ satisfying $\log^2 n\leq \ell \leq 2\log^2 n$.
	Let~$D$ be a $(\delta, n^{-\frac{1}{2}})$-almost regular robust $(\nu, \tau)$-outexpander on $n$ vertices. 
	Let~$M$ be a matching on $V(D)$. 
	Then, for any distinct $u,v\in V(D)$, there exist
	\[(1\pm n^{-\frac{1}{6}})\delta^\ell n^{\ell-1}\]
	$(u,v)$-paths $P$ of length $\ell$ such that $|V(P)\cap e|\leq 1$ for each $e\in M$.
\end{lm}

\begin{proof}
	By \cref{lm:regrob}, we can remove from $D$ at most $2\delta n^{\frac{3}{4}}$ inedges and at most $2\delta n^{\frac{3}{4}}$ outedges incident to each vertex to obtain a spanning subdigraph $D'\subseteq D$ which is a robust $(\frac{\nu}{2}, \tau)$-outexpander and is $(1-n^{-\frac{1}{4}})\delta n$-regular.
	
	Let $u,v\in V(D)$ be distinct.  
	Note that $D'$ contains at most as many $(u,v)$-paths of length $\ell$ such that $|V(P)\cap e|\leq 1$ for each $e\in M$ as $D$.
	Moreover, the number of $(u,v)$-paths of length $\ell$ in $D$ which use at least one edge from $D\setminus D'$ is at most $\ell \cdot 2\delta n^{\frac{3}{4}}\cdot (1+n^{-\frac{1}{2}})^{\ell-2}\delta^{\ell-2}n^{\ell-2} \leq  \delta^{\ell} n^{\ell-\frac{6}{5}}$.
	Therefore, it is enough to show that~$D'$ contains 
	$(1\pm 2n^{-\frac{1}{5}}) \delta^\ell n^{\ell-1}$
	$(u,v)$-paths $P$ of length $\ell$ such that $|V(P)\cap e|\leq 1$ for each $e\in M$.
	
	First, 
	we claim that 
	$D'$ contains $(1\pm n^{-\frac{1}{5}}) \delta^\ell n^{\ell-1}$ $(u,v)$-walks of length~$\ell$.
	Note that $D'$ contains exactly $(1- n^{-\frac{1}{4}})^\ell\delta^\ell n^\ell$ walks of length $\ell$ starting at~$u$.
	Let $k\coloneqq \lceil 2\nu^{-1}\rceil+1$ and let $(X_t)_{t\geq 0}$ be the Markov chain corresponding to a simple random walk on $D'$ starting at $u$ and define $(Y_t)_{t\geq 0}\coloneqq (X_{\ell-k\lfloor\frac{\ell}{k}\rfloor+kt})_{t\geq 0}$. Note that
	$\mathbb{P}[X_\ell=v]=\mathbb{P}[Y_{\lfloor\frac{\ell}{k}\rfloor}=v].$
	Then, \cref{lm:mixing,prop:markovrob}%
		\COMMENT{$\alpha\geq \frac{\nu^{k-1}\delta^{-k}}{2^{k-1}}$ and $\beta \leq \delta^{-1}$, so $2+2\alpha^{-1}\log \beta \leq 2+2^k\nu^{-k+1}\delta^k\log(\delta^{-1})\leq \log n\leq \lfloor\frac{\ell}{k}\rfloor$.}%
		\COMMENT{(with $(Y_t)_{t\geq 0}$ and $\lfloor\frac{\ell}{k}\rfloor$ playing the roles of $(X_t)_{t\geq 0}$ and $t$ in \cref{lm:mixing} and with $\ell-k\lfloor\frac{\ell}{k}\rfloor$ and $\frac{\nu}{2}$ playing the roles of $c$ and $\nu$ in \cref{prop:markovrob})}
	(applied with $\frac{\nu}{2}$ playing the role of $\nu$)
	imply that the number of $(u,v)$-walks of length $\ell$ in $D'$ is equal to%
		\COMMENT{For the last equality, note that
		\[1-\frac{n^{-\frac{1}{5}}}{3}\leq 1-2\ell n^{-\frac{1}{4}}\leq (1-n^{-\frac{1}{4}})^\ell\leq 1\]
		and
		\[\Big(1-\frac{\nu^{k-1}\delta^{-k}}{2^k}\Big)^{\lfloor\frac{\ell}{k}\rfloor}\leq \Big(1-\frac{\nu^{k-1}\delta^{-k}}{2^k}\Big)^{\frac{\nu\log^2 n}{3}}\leq \frac{n^{-\frac{1}{5}}}{3},\]
		where the last inequality holds since $\frac{\nu}{3}\log^2 n\log(1-\frac{\nu^{k-1}\delta^{-k}}{2^k})\leq -\frac{1}{5}\log n-\log 3$.
	 	Thus, \[\Big(1\pm\Big(1-\frac{\nu^{k-1}\delta^{-k}}{2^k}\Big)^{\lfloor\frac{\ell}{k}\rfloor}\Big)(1- n^{-\frac{1}{4}})^\ell
	 	=\Big(1\pm \frac{n^{-\frac{1}{5}}}{3}\Big)^2=(1\pm n^{-\frac{1}{5}}).\]}
	\begin{align*}
		\mathbb{P}[Y_{\lfloor\frac{\ell}{k}\rfloor}=v]\cdot (1- n^{-\frac{1}{4}})^\ell \delta^\ell n^\ell&= \Big(1\pm \Big(1-\frac{\nu^{k-1}\delta^{-k}}{2^k}\Big)^{\lfloor\frac{\ell}{k}\rfloor}\Big)n^{-1}\cdot (1- n^{-\frac{1}{4}})^\ell\delta^\ell n^\ell\\
		&= (1\pm n^{-\frac{1}{5}}) \delta^\ell n^{\ell-1},
	\end{align*}
	 as desired.

	The number of $(u,v)$-walks which are not paths $P$ such that $|V(P)\cap e|\leq 1$ for each $e\in M$ is at most $2\ell^2 n^{\ell-2}	\leq n^{\ell-\frac{3}{2}}$.
	This implies that $D'$ contains $(1\pm 2n^{-\frac{1}{5}}) \delta^\ell n^{\ell-1}$
	$(u,v)$-paths $P$ of length~$\ell$ such that $|V(P)\cap e|\leq 1$ for each $e\in M$, as desired. 
\end{proof}

\subsection{Ratio estimates}

In this section, 
we prove our main tool (\cref{lm:ratio}).
Note that we only need the case where $N$ is a matching to prove \cref{thm:edge,thm:main}. 
The case where $N$ is a regular spanning subgraph can be proved with the same arguments, so we state (and prove) the next \lcnamecref{lm:ratio} for both cases. (See \cref{sec:conclusion} for some applications of the regular case.)

\begin{lm}\label{lm:ratio}
	Let $0<\frac{1}{n} \ll \nu\leq \tau\ll \delta \leq 1$ and let $r\leq n^{\frac{1}{3}}$ be a positive integer.
	Let $G$ be a $\delta n$-regular robust $(\nu, \tau)$-expander on $n$ vertices and suppose that $N$ is a matching\footnote{We sometimes treat a matching as a graph and hence $e(N)$ refers to the number of edges in $N$.} in $G$ or a spanning $r$-regular subgraph of $G$. 
	For any non-negative integer $k$, let $\cM_k$ be the set of perfect matchings in $G$ which contain precisely $k$ edges of $N$. 
	Then, for each $k\in [\min\{e(N), n^{\frac{1}{3}}\}]$, we have
	\begin{align*}
	\frac{|\cM_k|}{|\cM_{k-1}|}= (1\pm n^{-\frac{1}{7}}) \frac{f(N,k)}{k\delta n}, \qquad \text{where }
	f(N,k)\coloneqq
	\begin{cases}
		 e(N)-(k-1)& \text{if $N$ is a matching};\\
		e(N) & \text{otherwise.}
	\end{cases}
	\end{align*}
\end{lm}

\begin{proof}
	Let $\ell\coloneqq \lceil\log^2 n\rceil+1$.
	Let $H$ be the auxiliary bipartite graph on vertex classes $\cM_k$ and $\cM_{k-1}$ defined as follows. Perfect matchings $M\in \cM_k$ and $M'\in \cM_{k-1}$ are adjacent in $H$ if and only if there exists $e\in E(N)\cap M$ such that $M\triangle M'$ is a cycle of length $2\ell$ whose intersection with $E(N)$ consists of the edge $e$.
	
	\begin{claim}\label{claim:ratiok-1}
		Each $\NEW{M'} \in \cM_{k-1}$ satisfies
		\[d_H(\NEW{M'})=(1\pm 3n^{-\frac{1}{6}})\delta^{\ell-1}n^{\ell-2}f(N,k).\]
	\end{claim}

	\begin{proofclaim}
		Let $\NEW{M'}\in \cM_{k-1}$ and $e=uv\in E(N)\setminus \NEW{M'}$ such that $u,v\notin V(E(N)\cap \NEW{M'})$. Denote by $w$ the (unique) neighbour of $u$ in $\NEW{M'}$. Then, note that the number of cycles $C$ of length $2\ell$ which alternate between edges of $\NEW{M'}$ and edges of $E(G)\setminus \NEW{M'}$ and which satisfy $E(C)\cap E(N)=\{e\}$ is precisely the number of $(w,v)$-paths of length $2\ell-2$ which alternate between edges of $E(G)\setminus (E(N)\cup \NEW{M'})$ and edges of $\NEW{M'}\setminus E(N)$.
		The number of such paths can be estimated using \cref{lm:countpaths} as follows.
		
		\NEW{Proceed as follows to} define an auxiliary digraph $D$ on $V(G)\setminus V(E(N)\cap \NEW{M'})$ \NEW{whose paths of length $\ell-1$ will roughly correspond to alternating paths of length $2\ell-2$ in $G$}. Denote $m\coloneqq |V(D)|=n-2(k-1)$ and, for any $x\in V(D)$, let $N_D^+(x)\coloneqq N_\NEW{M'}(N_{G-N-\NEW{M'}}(x)\cap V(D))$, where $G-N-\NEW{M'}$ refers to the graph obtained from $G$ by deleting $E(N)$ and $\NEW{M'}$ from $E(G)$. 
		Thus, each $x\in V(D)$ satisfies
		\begin{align*}
			d_D^+(x)&= |N_{G-N-\NEW{M'}}(x)\cap V(D)|=d_G(x)\pm (2k+r)
			=(1 \pm m^{-\frac{1}{2}})\delta m.
		\end{align*}
		Similarly, note that each $x\in V(D)$ satisfies $N_D^-(x)= N_{G-N-\NEW{M'}}(N_\NEW{M'}(x))\cap V(D)$ and so $d_D^-(x)=(1 \pm m^{-\frac{1}{2}})\delta m$. That is, $D$ is $(\delta, m^{-\frac{1}{2}})$-almost regular.
		Moreover, \NEW{we claim that $D$ is a robust $(\frac{\nu}{2},\NEW{2}\tau)$-outexpander. Indeed, let $S\subseteq V(D)$ satisfy $\tau n\leq 2\tau m\leq |S|\leq (1-2\tau)m\leq (1-\tau)n$. Then, for any $xy\in M'$ with $x\in RN_{\nu,G}(S)$ and $y\in V(D)$, the definition of $D$ implies that $|N_D^-(y)\cap S|=|N_{G-N-M'}(x)\cap S|\geq \nu n-(r+1)\geq \frac{\nu m}{2}$ and so $y\in RN_{\frac{\nu}{2},D}^+(S)$. That is,
		\begin{align*}
			|RN_{\frac{\nu}{2},D}^+(S)|\geq |N_\NEW{M'}(RN_{\nu,G}(S))\cap V(D)|\geq |S|+\nu n-(n-m)\geq |S|+\frac{\nu m}{2}
		\end{align*}}
		and so $D$ is a robust $(\frac{\nu}{2},\NEW{2}\tau)$-outexpander, \NEW{as claimed}.
		Thus, \cref{lm:countpaths} (applied with $m, \ell-1, \NEW{M'}\setminus E(N),\frac{\nu}{2}$, and \NEW{$2\tau$} playing the roles of $n, \ell, M, \nu$, and \NEW{$\tau$}) implies that $D$ contains%
			\COMMENT{Note that $m=(1\pm 2n^{-\frac{2}{3}})n$, so 
			\[m^{\ell-2}=\left(1\pm 2n^{-\frac{2}{3}}\right)^{\ell-2}n^{\ell-2}=\left(1\pm 4(\ell-2)n^{-\frac{2}{3}}\right)n^{\ell-2}=\left(1\pm n^{-\frac{1}{3}}\right) n^{\ell-2}.\]
			Moreover,
			\[m^{-\frac{1}{6}}\leq (n-2 n^{-\frac{2}{3}})^{-\frac{1}{6}}\leq \frac{3n^{-\frac{1}{6}}}{2} .\]
			Therefore,
			\[(1\pm m^{-\frac{1}{6}}) m^{\ell-2}= \left(1\pm \frac{3n^{-\frac{1}{6}}}{2}\right)\left(1\pm n^{-\frac{1}{3}}\right) n^{\ell-2}=\left(1\pm 2n^{-\frac{1}{6}}\right).\]}
		\[(1\pm m^{-\frac{1}{6}}) \delta^{\ell-1} m^{\ell-2}=(1\pm 2n^{-\frac{1}{6}}) \delta^{\ell-1}n^{\ell-2}\]
		$(w,v)$-paths $P$ of length $\ell-1$ such that $|V(P)\cap e'|\leq 1$ for each $e'\in \NEW{M'}$.
		By construction, there is a bijection between the $(w,v)$-paths of length $2\ell-2$ in $G$ which alternate between edges of $E(G)\setminus (E(N)\cup \NEW{M'})$ and edges of $\NEW{M'}\setminus E(N)$, and the $(w,v)$-paths $P$ of length $\ell-1$ in $D$ which satisfy $|V(P)\cap e'|\leq 1$ for each $e'\in \NEW{M'}\setminus E(N)$.
		
		If $N$ is a matching, then there are exactly $e(N)-(k-1)=f(N,k)$ edges $e=uv\in E(N)\setminus \NEW{M'}$ such that $u,v\notin V(E(N)\cap \NEW{M'})$, so the claim holds. If $N$ is a spanning $r$-regular subgraph of $G$, then there are at most $e(N)$ such edges and at least $e(N)-2rk$ such edges, so
		\begin{align*}
			d_H(\NEW{M'})= (1\pm 2n^{-\frac{1}{6}})\delta^{\ell-1}n^{\ell-2}\cdot (1\pm 4kn^{-1})e(N)=(1\pm 3n^{-\frac{1}{6}})\delta^{\ell-1}n^{\ell-2} f(N,k),
		\end{align*}
		as desired.
	\end{proofclaim}

	\begin{claim}\label{claim:ratiok}
		Each $M\in \cM_k$ satisfies
		\[d_H(M)= (1\pm 3n^{-\frac{1}{6}}) k\delta^\ell n^{\ell-1}.\]
	\end{claim}

	\begin{proofclaim}
		Let $M\in \cM_k$ and $e=uv\in E(N)\cap M$. Then, for any $w\in N_G(v)\setminus V(E(N)\cap M)$, one can show using the same arguments as in \cref{claim:ratiok-1}
		that $G$ contains
		$(1\pm 2n^{-\frac{1}{6}}) \delta^{\ell-1}n^{\ell-2}$
		$(u,w)$-paths of length $2\ell-2$ which alternate between edges of $E(G)\setminus (E(N)\cup M)$ and $M\setminus E(N)$. Since $|E(N)\cap M|=k$ and $\delta n-2k\leq |N_G(v)\setminus V(E(N)\cap M)|\leq \delta n$, we have%
			\COMMENT{Note that \[|N_G(v)\setminus V(N\cap M)|=\big(1\pm \frac{2k}{\delta n}\Big)\delta n=\Big(1\pm \frac{2n^{-\frac{2}{3}}}{\delta}\Big)\delta n=(1\pm n^{-\frac{1}{3}})\delta n.\]
			Moreover,
			\[(1\pm 2n^{-\frac{1}{6}})(1\pm n^{-\frac{1}{3}})=(1\pm 3n^{-\frac{1}{6}}).\]}
		\[d_H(M)= (1\pm 2n^{-\frac{1}{6}})\delta^{\ell-1} n^{\ell-2}\cdot k\cdot (1\pm 2k\delta^{-1}n^{-1})\delta n=(1\pm 3n^{-\frac{1}{6}})k\delta^\ell n^{\ell-1},\]
		as desired.
	\end{proofclaim}
 	
 	Since $\sum_{M\in \cM_k}d_H(M)=e(H)=\sum_{M\in \cM_{k-1}}d_H(M)$, \cref{claim:ratiok,claim:ratiok-1} imply that
 	\begin{align*}
 		\frac{|\cM_k|}{|\cM_{k-1}|}&=\frac{1\pm 3n^{-\frac{1}{6}}}{1\pm 3n^{-\frac{1}{6}}}\cdot  \frac{f(N,k)}{k\delta n}
 		=(1\pm n^{-\frac{1}{7}}) \frac{f(N,k)}{k\delta n}
 	\end{align*}
 	and so the \lcnamecref{lm:ratio} follows.	
\end{proof}

\subsection{Proof of Theorems \ref{thm:edge} and \ref{thm:main}}

We first derive \cref{thm:edge} and then \cref{thm:main}, where we also use~\cref{thm:edge}.
Observe that both proofs are very elementary given Lemma~\ref{lm:ratio}.
The (short) proof of \cref{thm:main} is similar to the analogue proof in~\cite{spiro2022counting}.

\begin{proof}[Proof of \cref{thm:edge}]
	Let $\cM_0$ be the set of perfect matchings in $G$ which do not contain~$e$ and~$\cM_1$ be the set of perfect matchings of $G$ which contain $e$.
	Then, \cref{lm:ratio} (applied with $\{e\}, \frac{d}{n}$, and $1$ playing the roles of $N,\delta$, and $k$) gives that
	\[\mathbb{P}[e\in M]=\frac{|\cM_1|}{|\cM_0|+|\cM_1|}=\frac{1\pm n^{-\frac{1}{7}}}{1+(1\pm n^{-\frac{1}{7}})d^{-1}}d^{-1}=(1\pm n^{-\frac{1}{8}})d^{-1},\]
	as desired.
\end{proof}

\begin{proof}[Proof of \cref{thm:main}]
	\NEW{Denote $\lambda\coloneqq d^{-1}|N|$ and let $Y\sim \Po(\lambda)$.
	For any non-negative integer $k$, let $p_k\coloneqq \bP[X=k]$ and define $\omega\coloneqq n^{\frac{1}{28}}$.
	Note that if $|N|<\omega^5$, then
	the union bound and \cref{thm:edge} imply that
	\[p_0\geq 1-\sum_{e\in N}\mathbb{P}[e\in M]\geq 1-(1+o_n(1))d^{-1}|N|\geq 1-\frac{\omega^6}{n}\geq 1-\omega^{-2}\]
	and, in particular, $p_k\leq \omega^{-2}$ for all positive integers $k$. On the other hand, $e^{-\lambda}\geq 1-\omega^{-1}$, while for all $k\in [\omega]$, we have $\frac{\lambda^k}{k!}e^{-\lambda}\leq \lambda \leq \omega^{-2}$ and, by Markov's inequality, $\mathbb{P}[Y\geq \omega]\leq \lambda\omega^{-1}\leq \omega^{-1}$.
	Therefore,
	\begin{align*}
		d_{\rm TV}(X,Y)&= \frac{1}{2}\sum_{k\geq 0}\left|p_k- \frac{\lambda^k}{k!}e^{-\lambda}\right|
		=\frac{1}{2}\left(|p_0-e^{-\lambda}|+\sum_{k=1}^\omega\left|p_k- \frac{\lambda^k}{k!}e^{-\lambda}\right|+\sum_{k>\omega}\left|p_k- \frac{\lambda^k}{k!}e^{-\lambda}\right|\right)\\
		&\leq\frac{1}{2}(\omega^{-2}+\omega\cdot \omega^{-2}+(\omega^{-2}+\omega^{-1}))\leq 3\omega^{-1}.
	\end{align*}
	We may therefore assume that $|N|\geq \omega^{5}$.}
	Observe that $\sum_{k>\omega}\frac{\lambda^k}{k!}\leq e^\lambda \omega^{-1}$%
		\COMMENT{$\sum_{k\geq \omega}\frac{\lambda^k}{k!}\leq\sum_{k\geq \omega}(\lambda e \omega^{-1})^k\leq  \sum_{k\geq \omega}(e \omega^{-1})^k =\frac{(e\omega^{-1})^\omega}{1-e\omega^{-1}}\leq e^\lambda \omega^{-1}$.}.
	Moreover, \cref{thm:edge} and linearity of expectation yield $\mathbb{E}[X]=(1+o_n(1))d^{-1}|N|$ and 
	so Markov's inequality implies that $\mathbb{P}[X> \omega] \leq  \omega^{-\frac{1}{2}}$.
	\NEW{For $k\in [\omega]$,} \cref{lm:ratio} implies that $\frac{p_k}{p_{k-1}}= (1\pm \omega^{-3})\frac{\lambda}{k}$ and so $p_k= (1\pm \omega^{-1})\frac{\lambda^k}{k!}p_0$.
	Note that 
	\begin{align*}
		1
		=\sum_{k\geq 0}p_k
		= (1 \pm 2\omega^{- \frac{1}{2}}) \sum_{k=0}^\omega p_k
		= (1 \pm 3\omega^{-\frac{1}{2}}) p_0\sum_{k=0}^\omega\frac{\lambda^k}{k!}
		= (1 \pm 4\omega^{- \frac{1}{2}}) p_0 e^{\lambda}.
	\end{align*}
	Hence $p_0= (1 \pm 5 \omega^{-\frac{1}{2}}) e^{-\lambda}$ and so $p_k=(1 \pm 6 \omega^{-\frac{1}{2}})\frac{\lambda^k}{k!}e^{-\lambda}$ for each $k\leq \omega$.
	This implies that $|p_k-\frac{\lambda^k}{k!}e^{-\lambda}|\leq 6 \omega^{-\frac{1}{2}}\frac{\lambda^k}{k!}e^{-\lambda}$ for each $k\leq \omega$.

	By Markov's inequality, we conclude that $\bP[X\geq \omega]\leq \omega^{- \frac{1}{2}}$ and $\bP[Y\geq \omega]\leq \omega^{- \frac{1}{2}}$.
	Therefore,
	\begin{align*}
		d_{\rm TV}(X, Y)
		= \frac{1}{2}\sum_{k\geq 0}\left|p_k- \frac{\lambda^k}{k!}e^{-\lambda}\right|
		\leq 3\omega^{- \frac{1}{2}}\sum_{k= 0}^\omega \frac{\lambda^k}{k!}e^{-\lambda} + \omega^{-\frac{1}{2}}
		\leq\omega^{- \frac{1}{3}},
	\end{align*}
	as desired.
\end{proof}

\section{Conclusion}\label{sec:conclusion}

In this paper we showed that uniformly chosen perfect matchings in robust expanders contain each edge asymptotically equally likely.
In fact, for a larger set of disjoint edges, these events are approximately independent.
As robust expanders are a fairly large class of graphs,
this in particular contains graphs $G$ on $n$ vertices with $\delta(G)\geq (\frac{1}{2}+o_n(1))n$, 
which confirms a question of Spiro and Surya~\cite{spiro2022counting} in a strong form.

\subsection{Regular subgraphs}

Spiro and Surya~\cite{spiro2022counting} also suggest to estimate the probability that a uniformly chosen perfect matching of Tur\'an graphs intersects a fixed spanning $r$-regular subgraph.
We prepared \cref{lm:ratio} for this question and hence we also have the following result.

\begin{thm}\label{thm:reg}
	For any $\delta> 0$, there exists $\tau >0$ such that for all $\nu>0$, there exists $n_0\in \mathbb{N}$ for which the following holds. 
	Let $n\geq n_0$ be even, $d\geq \delta n$, and let $r\leq n^{\frac{1}{50}}$ be a positive integer. 
	Then, for any $d$-regular robust $(\nu, \tau)$-expander $G$ on $n$ vertices, $M\sim U(\cP(G))$, any spanning $r$-regular subgraph~$N$ in $G$, $X\coloneqq |M\cap E(N)|$, and $Y\sim \Po(\frac{rn}{2d})$, we have 
	$d_{\rm TV}(X,Y)=o_n(1)$. 
\end{thm}

\begin{proof}
	Let $\lambda\coloneqq \frac{rn}{2d}$ and $\omega\coloneqq n^{\frac{1}{45}}$.
	The proof is almost the very same as the proof of \cref{thm:main} with $\omega$ is replaced by $\omega\lambda$ in most places.
	
	Markov's inequality and \cref{thm:edge} imply $\mathbb{P}[X> \omega\lambda] \leq  (\omega\lambda)^{-\frac{1}{2}}$.
	For all non-negative integers, define $p_k\coloneqq \bP[X=k]$. For \NEW{$k\in  [\omega\lambda]$}, we conclude by
	\cref{lm:ratio} that $\frac{p_k}{p_{k-1}}= (1\pm (\omega\lambda)^{-3})\frac{\lambda}{k}$ and so $p_k= (1\pm (\omega\lambda)^{-1})\frac{\lambda^k}{k!}p_0$.
	Similarly as above, we obtain $p_0= (1 \pm 5 (\omega\lambda)^{-\frac{1}{2}}) e^{-\lambda}$ and thus $p_k= (1 \pm 6 (\omega\lambda)^{-\frac{1}{2}})\frac{\lambda^k}{k!}e^{-\lambda}$ for each $k\leq \omega \lambda$.
	This implies that $|p_k-\frac{\lambda^k}{k!}e^{-\lambda}|\leq 6 (\omega\lambda)^{-\frac{1}{2}}\frac{\lambda^k}{k!}e^{-\lambda}$ for each $k\leq \omega \lambda$.
	
	By Markov's inequality, we conclude that $\bP[X> \omega\lambda]\leq \omega^{- \frac{1}{2}}$ and $\bP[Y> \omega\lambda]\leq \omega^{- \frac{1}{2}}$.
	Therefore,
	\begin{align*}
		d_{\rm TV}(X, Y)
		= \frac{1}{2}\sum_{k\geq 0}\left|p_k- \frac{\lambda^k}{k!}e^{-\lambda}\right|
		= 3 (\omega\lambda)^{-\frac{1}{2}}\sum_{k= 0}^{\omega\lambda}\frac{\lambda^k}{k!}e^{-\lambda} + \omega^{-\frac{1}{2}}
		\leq\omega^{- \frac{1}{3}},
	\end{align*}
	which completes the proof.
\end{proof}

As a corollary, one can calculate the probability that $r$ perfect matchings, each chosen independently and uniformly at random, are (edge-)disjoint. 
This relates to a problem of Ferber, H\"anni, and Jain \cite{ferber2020probability}, 
which asks for the probability of selecting $r$ edge-disjoint copies of a graph~$H$ in a host graph $G$. 
They answer this question for Hamilton cycles in the complete graph. 
The following \lcnamecref{cor:r} is an analogue for perfect matchings in the more general class of robust expanders. The proof follows immediately from \cref{thm:reg} by induction on $r$.

\begin{cor}\label{cor:r}
	For any $\delta> 0$, there exists $\tau >0$ such that for all $\nu>0$, there exists $n_0\in \mathbb{N}$ for which the following holds. 
	Let $n\geq n_0$ be even, $d\geq \delta n$, and $r\leq n^{\frac{1}{50}}$. 
	Then, for any $d$-regular robust $(\nu, \tau)$-expander $G$ on $n$ vertices and independent $M_1,\dots, M_r\sim U(\cP(G))$, we have
	\[\mathbb{P}[M_1,\ldots,M_r \text{ are disjoint}]= (1+o_n(1))e^{-\frac{n}{2d}{r\choose 2}}.\]
\end{cor}

\COMMENT{\begin{proof}
	We proceed by induction on $r$. The case $r=1$ is vacuously true.
	The case $r=2$ follows immediately from \cref{thm:reg}. Let $M_1,\dots, M_{r+1}\in U(\cP(G))$. Denote by $\cE$ the event that $M_i\cap M_j=\emptyset$ for all $1\leq i< j\leq r$ and define $N\coloneqq \bigcup_{i\in [r]}M_i$.
	Suppose inductively that
	\[\mathbb{P}[\cE]= (1+o_n(1))e^{-\frac{n}{2d}{r\choose 2}}.\]
	Then, \cref{thm:reg} implies that
	\begin{align*}
		\mathbb{P}[\forall 1\leq i< j\leq r+1 \colon M_i\cap M_j=\emptyset]&=\mathbb{P}[\cE]\mathbb{P}[M_r\cap N=\emptyset\mid \cE]\\
		&=(1+o_n(1))e^{-\frac{n}{2d}{r\choose 2}}\cdot (1+o_n(1))e^{-\frac{rn}{2d}}\\
		&=(1+o_n(1))e^{-\frac{n}{2d}{r+1\choose 2}},
	\end{align*}
	as desired.
\end{proof}}

\subsection{Bipartite graphs}

Of particular interest are perfect matchings in (balanced) bipartite graphs, but bipartite graphs are not robust expanders as the neighbourhood of one of the partition classes is only at most as large as the class itself.
However, the notion of robust expanders can be adapted to bipartite graphs.
Let $G$ be a bipartite graph with vertex partition $(A,B)$ and $|A|=|B|=n$.
We say that $G$ is a \emph{bipartite robust $(\nu, \tau)$-expander} if $|RN_{\nu,G}(S)| \geq  |S| + \nu n$ for each 
$S\subseteq A$ satisfying $\tau n\leq |S| \leq  (1 - \tau)n$.

The following is an analogue of \cref{thm:edge,thm:main,thm:reg} for bipartite graphs. 
This then also includes an approximation for the number of derangements.

\begin{thm}\label{thm:bipartite}
	For any $\delta> 0$, there exists $\tau >0$ such that for all $\nu>0$, there exists $n_0\in \mathbb{N}$ for which the following holds. 
	Let $n\geq n_0$, $d\geq \delta n$, and $r\leq n^{\frac{1}{50}}$. Let $G$ be a balanced bipartite $d$-regular robust $(\nu, \tau)$-expander on $2n$ vertices and suppose that $N$ is a matching in $G$ or a spanning $r$-regular subgraph of $G$.	
	Let $M\sim U(\cP(G))$, let $X\coloneqq |M\cap E(N)|$, let $Y\sim \Po(d^{-1}e(N))$, and let $e\in E(G)$. Then, $\bP[e\in M]= (1+o_n(1))d^{-1}$ and
	$d_{\rm TV}(X,Y)=o_n(1)$. 
\end{thm}

This result can be obtained by almost the very same proof. The only adaptation needed is when constructing the auxiliary digraph $D$ in \cref{lm:ratio}, where the vertex set of $D$ is now one partition class of $G$ minus $V(M\cap E(N))$. This digraph $D$ is then a robust outexpander in the sense as before.%
	\COMMENT{Formally, define $D$ as follows.\\
	Let $M\in \cM_{k-1}$ and $e=uv\in E(N)\setminus M$ such that $u,v\notin V(E(N)\cap M)$. Denote by $w$ the (unique) neighbour of $u$ in $M$ and $A$ the partition class of $G$ which contains $w$.	
	Define an auxiliary digraph $D$ on $A\setminus V(E(N)\cap M)$ as follows. Denote $m\coloneqq |V(D)|=n-2(k-1)$ and, for any $x\in V(D)$, let $N_D^+(x)\coloneqq N_M(N_{G-N-M}(x)\cap V(D))$. By the same arguments as in \cref{lm:ratio}, $D$ is a $(\delta,m^{\frac{1}{2}})$-almost regular robust $(\frac{\nu}{2},\tau)$-outexpander.}

\subsection{Almost regular graphs}

One can also adapt our proofs to graphs that are only almost regular, say $(\delta,f)$-regular.
Then, clearly, a statement similar to \cref{thm:edge} has to include an error term that depends on $f$.
If $f$ is small, then the conclusion of \cref{thm:edge} can stay the same.
On the other hand, if $f$ is large, say a small fixed constant, then
the error term becomes $(1\pm \gamma)$ for some $\gamma>0$ that is a function of $f$ which tends to $0$ as $f$ tends to $0$. 
Observe that in this case, the length of paths considered for example in \cref{lm:countpaths} has to be reduced to a value that is independent of $n$ as otherwise the error term coming from the degree irregularity becomes too large.
For $\varepsilon$-super-regular bipartite graphs, such a result has been proven by the second author in \cite{kim2019blow}, using results of Alon, R\"odl, and Ruci\'nski \cite{alon1998perfect}. 

\section*{Acknowledgements}

We are grateful to the referees for helpful comments on an earlier version of this paper.

\bibliographystyle{abbrv}
\bibliography{bibliography.bib}

\end{document}

%% file: Preamble.tex
\usepackage[shortlabels]{enumitem}
\usepackage{xr-hyper, zref}
\usepackage[hypertexnames=false]{hyperref}
\usepackage{etoolbox}
\usepackage{bm}
\usepackage{amsthm,amssymb,amsmath}
\usepackage[labelformat=simple]{subcaption}

\usepackage{tikz,graphicx,placeins}
\usetikzlibrary{bending,decorations.pathmorphing}
\usetikzlibrary{shapes.geometric}
\usetikzlibrary{arrows}
\usetikzlibrary{arrows.meta}
\tikzset{>={Latex[width=2mm,length=2mm]}}
\usepackage{psfrag,mathrsfs}
\usepackage{mathtools}
\usepackage[foot]{amsaddr}
\usepackage[british]{babel}
\usepackage[babel]{microtype}
\usepackage[margin=2cm]{geometry}
\usepackage[capitalise]{cleveref}
\usepackage[textsize=footnotesize]{todonotes}
\usepackage{environ}
\usepackage{float}
\usepackage{bigfoot}
\usepackage[noadjust]{cite}
\usepackage{dsfont}
\usepackage{moreenum}
\usepackage{eqparbox}
\usepackage[nonumberlist]{glossaries}
\renewcommand{\glsglossarymark}[1]{}

\date{}
\address{Institut f\"{u}r Informatik, Universit\"{a}t Heidelberg, 69120 Heidelberg, Deutschland}

\author[Bertille~Granet]{Bertille Granet}
\email{granet@informatik.uni-heidelberg.de}

\author[Felix~Joos]{Felix Joos}
\email{joos@informatik.uni-heidelberg.de}

\thanks{The research leading to these results was partially supported by the Deutsche Forschungsgemeinschaft (DFG, German Research Foundation) -- 428212407
	and by the DFG under Germany’s Excellence Strategy EXC-2181/1 -- 390900948 (the Heidelberg STRUCTURES Cluster of Excellence).
}

\newcommand{\COMMENT}[1]{}

\newcommand{\NEW}[1]{#1}

\newcommand{\APPENDIX}[1]{}

\newcommand{\notinsubfile}[1]{}

\hypersetup{colorlinks=true,
	citecolor=blue,
	filecolor=blue,
	linkcolor=blue,
}

\SetEnumitemKey{longlabel}{wide=0.5cm, leftmargin=*, align=right}
\setlist[enumerate]{itemsep=5pt, topsep=5pt,leftmargin=1.5cm}
\setlist[itemize]{itemsep=5pt, topsep=5pt, label=--}

\newlist{steps}{enumerate}{1}
\setlist[steps,1]{
	label=\textbf{Step \arabic*:},
	ref=\arabic*, 
	wide,
	parsep=0pt,
	itemsep=10pt,
	topsep=10pt}
\Crefname{stepsi}{Step}{Steps}
\crefname{stepsi}{Step}{Steps}

\newlist{case}{enumerate}{1}
\setlist[case,1]{
	label=\textbf{Case \arabic*:},
	ref=\arabic*, 
	wide,
	parsep=0pt,
	itemsep=10pt,
	topsep=10pt}
\Crefname{casei}{Case}{Cases}
\crefname{casei}{Case}{Cases}

\AtBeginEnvironment{thm}{\setlist[enumerate,1]{label=\upshape(\roman*),ref=(\roman*)}}
\AtBeginEnvironment{lm}{\setlist[enumerate,1]{label=\upshape(\roman*),ref=(\roman*)}}
\AtBeginEnvironment{prop}{\setlist[enumerate,1]{label=\upshape(\roman*),ref=(\roman*)}}
\AtBeginEnvironment{cor}{\setlist[enumerate,1]{label=\upshape(\roman*),ref=(\roman*)}}
\AtBeginEnvironment{claim}{\setlist[enumerate,1]{label=\upshape(\roman*),ref=(\roman*)}}
\AtBeginEnvironment{fact}{\setlist[enumerate,1]{label=\upshape(\roman*),ref=(\roman*)}}

\creflabelformat{equation}{#2(#1)#3}
\crefdefaultlabelformat{#2\textup{#1}#3}

\Crefname{enumi}{}{}
\Crefname{thm}{Theorem}{Theorems}
\Crefname{lm}{Lemma}{Lemmas}
\Crefname{cor}{Corollary}{Corollaries}
\Crefname{prop}{Proposition}{Propositions}
\Crefname{claim}{Claim}{Claims}
\Crefname{equation}{}{}
\Crefname{conjecture}{Conjecture}{Conjectures}
\Crefname{figure}{Figure}{Figures}
\Crefname{fact}{Fact}{Facts}

\theoremstyle{definition}
\newtheorem{definition}{Definition}[section]

\theoremstyle{plain}
\newtheorem{claim}{Claim}

\newtheorem{thm}[definition]{Theorem}
\newtheorem{cor}[definition]{Corollary}
\newtheorem{lm}[definition]{Lemma}

\newtheorem*{lm*}{Lemma}

\NewEnviron{property}[1]{\begin{equation}
	\tag{\(#1\)}
	\parbox{0.9\linewidth}{{\itshape\BODY}}
	\end{equation}}

\newenvironment{proofclaim}[1][Proof of Claim]{\begin{proof}[#1]}{\end{proof}}

\newcounter{claimnumber}
\AtBeginEnvironment{proof}{\setcounter{claim}{0}}
\AtBeginEnvironment{proofclaim}{\setcounter{claimnumber}{\value{claim}}}
\AtEndEnvironment{proofclaim}{\setcounter{claim}{\value{claimnumber}}}

\numberwithin{equation}{section}

\renewcommand{\epsilon}{\varepsilon}

\DeclareMathOperator{\Po}{Po}

\DeclareMathOperator{\pma}{pm}
\newcommand{\bP}{\mathbb{P}}
\newcommand{\bE}{\mathbb{E}}

\newcommand{\cE}{\mathcal{E}}

\newcommand{\cM}{\mathcal{M}}

\newcommand{\cP}{\mathcal{P}}

%% file: main.bbl
\begin{thebibliography}{10}

\bibitem{alon1998perfect}
N.~Alon, V.~R{\"o}dl, and A.~Ruci{\'n}ski.
\newblock Perfect matchings in $\varepsilon$-regular graphs.
\newblock {\em Electron. J. Combin.}, 5:R13, 1998.

\bibitem{ferber2020probability}
A.~Ferber, K.~H\"anni, and V.~Jain.
\newblock The probability of selecting $k$ edge-disjoint {H}amilton cycles in
  the complete graph.
\newblock {\em arXiv:2001.01149}, 2020.

\bibitem{gruslys2021cycle}
V.~Gruslys and S.~Letzter.
\newblock Cycle partitions of regular graphs.
\newblock {\em Combin. Probab. Comput.}, 30:526--549, 2021.

\bibitem{johnston2022deranged}
D.~Johnston, P.~M. Kayll, and C.~Palmer.
\newblock Deranged matchings: proofs and conjectures.
\newblock {\em arXiv:2209.11319}, 2022.

\bibitem{joos2022fractional}
F.~Joos and M.~K{\"u}hn.
\newblock Fractional cycle decompositions in hypergraphs.
\newblock {\em Random Structures Algorithms}, 61:425--443, 2022.

\bibitem{kahn1998random}
J.~Kahn and J.~H. Kim.
\newblock Random matchings in regular graphs.
\newblock {\em Combinatorica}, 18:201--226, 1998.

\bibitem{kim2019blow}
J.~Kim, D.~K{\"u}hn, D.~Osthus, and M.~Tyomkyn.
\newblock A blow-up lemma for approximate decompositions.
\newblock {\em Trans. Amer. Math. Soc.}, 371:4655--4742, 2019.

\bibitem{kuhn2015robust}
D.~K{\"u}hn, A.~Lo, D.~Osthus, and K.~Staden.
\newblock The robust component structure of dense regular graphs and
  applications.
\newblock {\em Proc. Lond. Math. Soc.}, 110:19--56, 2015.

\bibitem{kuhn2013hamilton}
D.~K{\"u}hn and D.~Osthus.
\newblock Hamilton decompositions of regular expanders: a proof of {K}elly’s
  conjecture for large tournaments.
\newblock {\em Adv. Math.}, 237:62--146, 2013.

\bibitem{spiro2022counting}
S.~Spiro and E.~Surya.
\newblock Counting deranged matchings.
\newblock {\em arXiv:2211.01872}, 2022.

\end{thebibliography}
